\documentclass[11pt]{amsart}
\usepackage{pgfplots,amssymb,amsmath,amsfonts,amsthm,nomencl,mathrsfs} 
\usepackage[arrow, matrix, curve]{xy}
\usepackage[latin1]{inputenc}
\usepackage{a4wide}
\usepackage{hyperref}

\newcommand{\IR}{\mathbb{R}}

\newcommand{\question}[1]{\leavevmode{\marginpar{\tiny%
$\hbox to 0mm{\hspace*{-0.5mm}$\leftarrow$\hss}%
\vcenter{\vrule depth 0.1mm height 0.1mm width \the\marginparwidth}%
\hbox to 0mm{\hss$\rightarrow$\hspace*{-0.5mm}}$\\\relax\raggedright #1}}}

\renewcommand{\O}{\mathrm{O}}

\renewcommand{\a}{\alpha}
\renewcommand{\b}{\beta}
\newcommand{\ct}{\gamma}
\renewcommand{\d}{\delta}
\newcommand{\e}{\epsilon}

\newcommand{\IBB}{\mathscr{B}}

\newcommand{\ICC}{\mathsf{C}^{\infty}}

\newcommand{\IL}{\mathsf{L}}

\newcommand{\Id}{{\rm d}}

\newcommand{\f}{\frac}
\newcommand{\nn}{\nonumber}

\newcommand {\rr}{\mathbb{R}}
\newcommand {\grad}{\mathrm{grad}}

\newcommand{\sect}{\mathrm{Sec}}
\newcommand{\ric}{\mathrm{Ric}}

\theoremstyle{plain}            
\newtheorem{theorem}{theorem}[section]
\newtheorem{Lemma}[theorem]{Lemma}
\newtheorem{Corollary}[theorem]{Corollary}
\newtheorem{Theorem}[theorem]{Theorem}

\theoremstyle{definition}

\newtheorem{Remark}[theorem]{Remark}

\newtheorem{Notation}[theorem]{Notation}

\allowdisplaybreaks[1]

\begin{document}

\begin{titlepage}
\title{Quantitative $\mathsf{C}^1$-estimates on manifolds}

\author[B. G\"uneysu]{Batu G\"uneysu}
\address{Batu G\"uneysu, Institut f\"ur Mathematik, Humboldt-Universit\"at zu Berlin, 12489 Berlin, Germany} \email{gueneysu@math.hu-berlin.de}

\author[S. Pigola]{Stefano Pigola}
\address{Stefano Pigola, Dipartimento di Scienza e Alta Tecnologia - Sezione di Matematica, Università dell'Insubria, 22100 Como, Italy} \email{stefano.pigola@uninsubria.it}

\end{titlepage}



\maketitle

\begin{abstract} We prove a $\mathsf{C}^1$-elliptic estimate of the form
\[
\sup_{B(x,r/4)} |\grad (\psi) | \leq \frac{C}{\min \Big( 1, r, (\| \mathrm{Sec}_g \|_{\mathsf{L}^{\infty}(B_g(x,r))} +\epsilon)^{-1/2} \Big)}  \left\{ \sup_{B(x,r/2)} | f | + \sup_{B(x,r/2)} |\psi| \right\},
\]
valid on any complete Riemannian manifold $M$ and for any smooth solution of the Poisson equation $\Delta \psi = f$ which is defined in a neighbourhood of the geodesic ball $B(x,r)$. Above, $C$ is a constant which only depends on $\dim(M)$ and $\epsilon >0$ is arbitrary. In case of global solutions, the estimate  is sensitive of the curvature growth on large balls and can be applied to deduce global results such as the zero-mean value property of $f$ as in the compact setting.
\end{abstract}

\section{Introduction and main result}

\setcounter{page}{1}

\subsection{Some notation}

Given a smooth Riemannian manifold $(M,g)$ we denote with
$$
\Id:\ICC(M)\longrightarrow  \Omega^1_{\ICC}(M)
$$
the usual exterior derivative, with $r_{\mathrm{inj},g}(x)$ the injectivity radius, and with $\mathrm{Sec}_g$ the sectional curvature, and with $\mathrm{Ric}_g$ the Ricci tensor. Likewise, $\Id_g(x,y)$ will denote the geodesic distance and $B_g(x,r)$ the induced open geodesic balls, and we write $\mu_g(\Id x)$ for the volume measure. All smooth fiber metrics that are induced by $g$ will be denoted with $(\bullet,\bullet)_g$, and the corresponding fiber norms with $|\bullet|_g$. It follows that for every smooth function $\psi:U\to\IR$ which is defined on some open subset $U\subset M$ one has $|\Id \psi|_g=|\mathrm{grad}_g (\psi) |_g$, where $\mathrm{grad}_g (\psi) \in \Gamma_{\ICC}(U, TM)$ denotes the Riemannian gradient. Locally, one has
\begin{align*}
&\Id\psi=\sum_i (\partial_i\psi )\Id x^i,\>\>\mathrm{grad}_g (\psi)=\sum_i \sum_j(g^{ij}\cdot\partial_j\psi) \partial_i,\>\>|\Id\psi|_g^2=\sum_{ij}g^{ij}\cdot\partial_i\psi\cdot\partial_j\psi.
\end{align*}
The divergence $\mathrm{div}_g(X)\in\mathsf{C}^{\infty}(U)$ of a smooth vector field $X:U\to TM$ is locally defined by
$$
\mathrm{div}_g(X)=\f{1}{\sqrt{\det(g)}}\sum_i\partial_i(\sqrt{\det(g)}X^i),\quad \text{ if $X=\sum_i X^i \partial_i$ in $U$.}
$$

The scalar (by convention negative-definite) Laplace-Beltrami operator will be denoted with 
$$
\Delta_g=-\Id^g\Id= \mathrm{div}_g\circ\mathrm{grad}_g.
$$
 Locally, 
$$
\Delta_g\psi=\f{1}{\sqrt{\det(g)}}\cdot\sum_{ij}\partial_i \left(\sqrt{\det(g)}\cdot g^{ij}\cdot\partial_j\psi\right),
$$
and in case the given coordinate system is (componentwise) $\Delta_g$-harmonic, the latter expression boils down to
$$
\Delta_g\psi=\sum_{ij}g^{ij}\cdot\partial_i\partial_j\psi.
$$

\subsection{Introductory observations}
In this paper we will be concerned with $\mathsf{C}^1$-elliptic estimates of the form
\[
\sup_{B_g(x,r/2)} |\Id\psi |_g \leq C \left\{ \sup_{B_g(x,r)} |\Delta_g \psi| + \sup_{B_g(x,r)} |\psi| \right\},
\]
for any smooth function $\psi$ on $M$, where we want to make explicit the dependence of the constant $C=C({B_g(x,r)})$ on the local geometry of $(M,g)$, and in particular on the radius $r$. As we will make clear later on (cf. Corollary \ref{anw} below), a precise control on these quantities is useful for many applications in geometric analysis, for explain, to deduce that the existence of a solution of the Poisson equation $\Delta_g \psi = f$ implies the zero-mean value property
$$
\int_{M} f \, \Id\mu _{g} =0
$$
even in the complete non-compact setting, provided  $\psi$, $f$ and the geometry of $(M,g)$ satisfy certain growth assumptions.\smallskip

We would like to stress that if one is only interested in a qualitative $\mathsf{C}^1$-elliptic estimate (without an explicit control of the constant $C$ above), it is straightforward to achieve such an estimate by locally applying classical interior Hölder estimates for elliptic operators.
Other possible approaches that lead to global qualitative bounds are based on iteration techniques and will be discussed in the next section when we will make a comparison with our result.\smallskip

A possible way to obtain a local quantitative $\mathsf{C}^1$-elliptic estimate would be to combine Cheng-Yau\rq{}s gradient estimate \cite{cheng} for positive harmonic functions with Green\rq{}s kernel estimates on balls. Very quickly: to a given smooth function $\psi$ on $B_{g}(x,2R_{0})$ we can associate the harmonic function $u$ defined by
\[
u(y) = \psi(y) + \int_{B_{g}(x,2R_{0})} \Delta_g \psi (z) G_{2R_{0}}(y,z)\Id \mu_g(z)
\]
where $G_{2R_{0}}(y,z)$ denotes the Green\rq{}s kernel of $B_{g}(x,2R_{0})$. If we apply the Cheng-Yau gradient estimate to the positive harmonic function $u_{\varepsilon}(y) = \sup_{B_g(x,R_{0})}|u|-u(y)+\varepsilon$ on the smaller ball $B_g(x,R_{0})$ and then we let $\varepsilon \to 0$ we obtain that
\[
\sup_{B_g(x,R_{0}/2)}|\Id u|_g \leq 2C(m,K,R_{0}) \sup_{B_g(x,R_{0})}|u|,
\]
where the constant $C$ is completely explicit. Whence, the validity of  the desired $\mathsf{C}^1$-estimate of $\psi$  follows provided we are able to produce a quantitative control of  the function
\[
y \mapsto \max\Big(\int_{B_{g}(p,2R_0)}|\Id_y G_{2R_0}(y,z)|_g \Id\mu_g( z)\>,\> \int_{B_{g}(p,2R_0)}G_{2R_0}(y,z) \Id \mu_g(z) \Big).
\]
We are going to prove the validity of a  $\mathsf{C}^1$-estimate where the constant depends explicitly on the ray and on the absolute bound of the sectional curvature on the corresponding ball, and via a direct method that does no require the use of the Green function.
\smallskip

Before we state our main result, we find it instructive to take a look first at the Euclidean case $M=\IR^m$ with its standard metric, in order to see how the optimal dependence on the constant on $r$ should look like: If $m=1$, it is straightforward to prove that, for every $a>0$, $0<r\leq 1$, and every smooth $u:I_{2r}(a)\to\IR$ one has 
\begin{equation}\label{elementary}
|u'(a)| \leq \frac{2}{r} \left\{ \sup_{I_{r}(a)} |u| + \sup_{I_{r}(a)} |u''|\right\},
\end{equation}
where $I_r(a):= (a-r,a+r)$. Indeed, set $\tilde{a}=a-r$, $\tilde{b}=a+r$ and define $v(x) = u(x\tilde{b} + (1-x)\tilde{a})$ for $x\in [0,1]$. Then, from the Taylor expansion with uniform reminder we have:
\[
v(x) = v(0) + v'(0) x + x^{2}\int_{0}^{1}(1-s)v''(xs)ds.
\]
By evaluating this latter at $x = 1/2$ and replacing the expression of $v$ we obtain:
\begin{align*}
u'(\tilde{a}) &=  \frac{2}{ \tilde{b}-\tilde{a} }\{ u \big(\tilde{b}/2+\tilde{a}/2\big) - u(\tilde{a})\} \\
 &- \frac{\tilde{b}-\tilde{a}}{2} \int_{0}^{1}(1-s)u''\big(s\tilde{b}/2+(2-s)\tilde{a}/2\big) ds.
\end{align*}
Whence we get the estimate
\[
| u'(\tilde{a}) | \leq \frac{4}{\tilde{b}-\tilde{a}} \max_{[\tilde{a},\tilde{b}]} |u| + \frac{\tilde{b}-\tilde{a}}{4} \max_{[\tilde{a},\tilde{b}]}|u''|.
\]
Recalling that $\tilde{a}=a-r$, $\tilde{b}=a+r$, the latter implies the validity of \eqref{elementary}.\\
For arbitrary dimensions $m \geq 2$, let $\psi: B(x,2R_0)\subset \rr^{m} \to \IR$ be a smooth function, $R_0>0$ fixed, 
$$
y=(y_1,\dots,y_m),\>x=(x_1,\dots,x_m)\in \IR^m\>\>\>\text{with $y\in B(x,R_0/2)$.}
$$
Applying \eqref{elementary} to 
$$
u:=\psi(y_1,\dots,y_{j-1},\bullet,y_{j+1},\dots,y_m)
$$
with  $a=y_j$ and $r= \min (R_0/2\sqrt{m},1)$, so that the closed $r$-cube around $y$ lies in $B(y,R_0/2)$, implies that, for some $C(m)>0$,
$$
|\Id\psi(y)| \leq \frac{C(m)}{  \min (R_0/2\sqrt{m},1)} \left\{\sup_{B(y,R_0/2)}|\psi|+\sup_{B(y ,R_0/2)}|\Delta\psi|\right\}.
$$
Therefore, using $B(y,R_0/2)\subset B(x,R_0)$, we have shown
\begin{align}\label{aposs}
\sup_{B(x,R_0/2)}|\Id\psi(y)|  \leq \frac{C(m)}{\min (R_0/2\sqrt{m},1)} \left\{\sup_{B(x,R_0)}|\psi|+\sup_{B(x,R_0)}|\Delta\psi|\right\}.
\end{align}

\subsection{Statement of the main result of the paper}
Let us now return to the general case. The following uniform inequality, which is valid on every complete Riemannian manifold, is the main result of this paper:

\begin{Theorem}\label{est}
For every  $m\geq 2$, there exists a universal constant $C=C(m)>0$, such that for
\begin{itemize}
\item [-] all smooth complete Riemannian manifolds $(M,g)$ with $\dim (M) = m$,
\item [-] all $A_{1},A_{2} >0$, $R_{0}> 0$, $x\in M$ with
\begin{equation}\label{curv-conditions2}
\sect_g\leq A_{1},\quad \mathrm{Ric}_g\geq -A_2\text{  on }B_g(x,R_{0}),
\end{equation}
\item [-] and all $\psi\in \ICC(B_{g}(x,R_{0}))$,
\end{itemize}
one has the bound
\begin{align}\label{uniform-estimate}
\sup_{B_{g}(x, R_0/4)}|\Id \psi|_g\leq \frac{C}{\min \Big( 1, \min(  \pi/\sqrt{A_{1}} , R_{0})/2 , \sqrt{1/A_2} \Big)} \Big(\sup_{ B_g(x,R_0/2)} \left|\Delta_g \psi\right|+\sup_{ B_g(x,R_0/2)}\left| \psi\right|\Big).
 \end{align}
\end{Theorem}

\begin{Remark}\label{est2}
The dependence on the geometry around the selected point can be made more transparent as follows, where we use the notation $T^{+} = \max(T,0)$ and $T^{-} = \min(T,0)$ for any real number $T$:\smallskip

Let $(M,g)$ be a smooth complete, $m$-dimensional Riemannian manifold. For every $R_0>0$, $x \in M$, let
 \[
 \sigma_{g}(x,R_0) :=  \sup_{B_{g}(x,R_0)} \sect_{g}\in \rr,\quad \rho_g(x,r) = \inf_{y\in B_{g}(x,R_0)}\min\mathrm{spec}(\ric_{g}(y))\in \rr.
 \]
 Then, for every $\epsilon>0$ we can define
 \[
 A_1:= \max( \sigma_g(x,R_0)^{+},\epsilon),\quad  A_2:=-\min(\rho_g(x,R_0)^{-}, -\epsilon)
 \]
so that one has the validity of \eqref{curv-conditions2}.
\end{Remark}

The very remarkable fact about  estimate \eqref{uniform-estimate} is that the constant depends completely explicitly on the local geometry of $(M,g)$. In addition, the estimate is of global nature in the sense that the behaviour of the constant is sensitive of the (sectional) curvature growth on large balls. This kind of result is not accessible by De Giorgi-Nash-Moser iteration techniques that typically produce very implicit constants with a huge dependence, say exponential or quadratic exponential, on the parameters. Moreover, since the Moser technique involves the Sobolev constant (that needs to be estimated), it gives rise naturally to results that are much more local than global in their nature. The obvious exception is the case of a uniform curvature bound: a uniform local estimate is in fact global. The reader may consult the book \cite{Li-book}, and the very recent preprint \cite{Zhang-Zhu} where elliptic estimates are obtained via iteration methods in the smooth metric measure space setting.\smallskip

On the other hand, the drawback of our technique is that it really needs a double-sided control on the sectional curvature, even when this latter may be assumed to be uniform. In contrast, as it is visible from the above references, the iteration technique requires just a lower uniform bound on the Ricci curvature. Thus, both techniques are relevant.
\smallskip

Let us make some comments on our proof of Theorem \ref{est}, which will be given in Section \ref{beweis}. Our approach has a very geometric nature and is as follows: Using the locally required curvature bounds, we construct a certain surjective local isometry 
\begin{align}\label{liss}
B_{\bar g}(\bar x, R)\longrightarrow B_{ g}( x, R),\>\>R:= \min(1/\sqrt{A_1},R_0)/2 ,
\end{align}
where $(\bar M, \bar g)$ is another Riemannian manifold satisfying $r_{\mathrm{inj},g}(\bar x) \geq R$, and $\bar x\in \bar M$, so that it essentially remains to prove an estimate for $B_{\bar g}(\bar x, R)$ with respect to $\bar g$.
In the latter case, we pick $\mathsf{W}^{1,p}$-harmonic coordinates around $\bar x$, where $\Delta_{\bar g}$ takes the simple form $\Delta_{\bar g}=\sum_{ij }\bar g^{ij}\partial_i\partial_j$. This makes it possible to use interior elliptic $\IL^q$-estimates in combination with the Sobolev embedding theorem to estimate the gradient of $\psi$ lifted to $B_{\bar g}(\bar x, R)$, where now the constant depends explicitely on the harmonic radius of $\bar g$ in $\bar x$. On the other hand, by a well-known result of M. Anderson and J. Cheeger, \cite{cheeger-a}, the latter can be controlled explicitely by the Ricci curvature and the injectivity radius of $\bar g$ in $\bar x$, and we can control the latter explicitely in terms of $A_1, A_2, R_0$. Let us point out that working directly in harmonic coordinates with respect to $g$ at $x$ without lifting everything, and following from thereon the same lines of thought as above, automatically leads to a dependence of the constant on the injectivity radius of $g$, which is a much weaker result. The idea of using a lifting procedure in order to avoid assumptions on the injectivity radius can be traced back to a paper of J. Cheeger and M. Gromov \cite{cheeger-g}, where it has been used to smoothing Riemannian distance functions. The use we make of this idea in the present paper inspires to the recent \cite{Carron-ASNS}, where applications to critical metrics, $\varepsilon$-regularity results and geometric rigidity are presented. The authors would like to thank G. Carron for having pointed out this reference. \vspace{2mm}

As an application of our main result we point out the following
\begin{Corollary}\label{anw}
Let $(M,g)$ be a smooth geodesically complete Riemannian manifold of dimensional $m=\dim (M)\geq 2$, and assume that there exist numbers $\a,\b\in [0,\infty)$ and a point $o\in M$ such that
 \begin{itemize}
\item [(a)] $\mu_g\big(B_g(o,2R) \setminus B_g(o,R)\big) = \O(R^{\a}\log(R))$ (as $R\to\infty$),
 \item [(b)] $\| \mathrm{Sec}_g \|_{\mathsf{L}^{\infty}(B_g(o,R))} = \O(R^{\b}\log(R))$.
\end{itemize}
Assume further that $u \in \mathsf{C}^{\infty}(M)$, that $\int_{M} \Delta_g u   \, \Id\mu _{g}$ exists\footnote{that is, if one has $(\Delta_g u)_+\in \mathsf{L}^1(M,g)$ or $(\Delta_g u)_-\in \mathsf{L}^1(M,g)$} and that there are numbers $ \ct,\d\in [0,\infty)$ such that
\begin{itemize}
 \item [(i)] $\| u \|_{\mathsf{L}^{\infty}(B_g(o,R))} = \O(R^{\ct}\log(R))$,
 \item [(ii)] $\| \Delta_g u \|_{\mathsf{L}^{\infty}(B_g(o,R))} = \O(R^{\d}\log(R))$,
 \item [(iii)] $\a+\b+\ct+\d<1$.
\end{itemize}
Then one has
\[
\int_{M} \Delta_g u   \, \Id\mu _{g} =0.
\]
\end{Corollary}
\begin{Remark} 1. Clearly, Corollary \ref{anw} can take the form of a non-existence result for global solutions of the Poisson equation $\Delta_{g} u = f$ on $M$. Indeed, under the growth condition (ii) on $f$ and  the growth assumptions (a), (b) on the geometry of $M$, if $\int_{M} f \Id\mu_{g} \not=0$ then there exist no smooth solution $u$ of $\Delta_{g} u = f$ satisfying $\|u\|_{\mathsf{L}^{\infty}(B_{g}(o,R))} = O(R^{\beta }\log R)$, for any $\beta$ as in (iii).\\
2. Condition (a) from Corollary \ref{anw} implies that $(M,g)$ has a sublinear volume growth at $\infty$, in the sense that for any $0< \e \ll 1$ such that $\e+ \a<1$ one has $\mu_g(B_g(o,R)) = \O(R^{\a + \e})$. To see this, observe that 
$$
\mu_g(B_g(o,2R)\setminus B_g(o, R)) = \O(R^{\a+\e}).
$$
Let $n \in \mathbb{N}$ be the largest integer such that $2^{n} < R$ so that $2^{n+1} \geq R$. Then
 \[
 \mu_g(B_g(o,R)) \leq \mu_g(B_g(o,2^{n+1})) = \O\Big(\sum_{j=0}^{n+1} (2^{\a+\e})^{j}\Big) = \O(2^{n(\a+\e)}) = \O(R^{\a + \e}).
 \]
 Conversely, condition (a) is satisfied, if there exists $\a$ such that for all large $R$ one has 
 $$
 \mathrm{const}_1\cdot R^{\a} \leq \mu_g(B_g(o,R)) \leq  \mathrm{const}_2 \cdot R^{\a}
 $$
\end{Remark}

\begin{proof}[Proof of Corollary \ref{anw}.] This result follows immediately from Theorem \ref{est} and the global divergence theorem by L. Karp \cite{karp}, which states that if $X$ is a locally integrable vector field on a smooth geodesically complete Riemannian manifold $(M',g')$ such that for some $o'\in M'$ one has 
$$
 \int_{B_{g'}(o,2R) \setminus B_{g'}(o,R)} | X |_{g'} \Id\mu_{g'}= \mathrm{o}(R),
$$
then the existence of $\int_{M'} \mathrm{div}_{g'}(X)\Id \mu_{g'} $ implies $\int_{M'} \mathrm{div}_{g'}(X)\Id \mu_{g'}=0$. Indeed, Theorem \ref{est} implies
 \[
 \| |\mathrm{grad}_g (u) \|_{\mathsf{L}^{\infty}(B_g(o,R))} = \| |\Id u|_g \|_{\mathsf{L}^{\infty}(B_g(o,R))} = \O (R^{\b+\ct+\d} \log^{\b+\ct+\d}(R) )
 \]
 In particular, using the volume assumption,
 \[
\frac{ \int_{B_g(o,2R) \setminus B_g(o,R)} |\mathrm{grad}_g u |_g \Id\mu_{g}}{R} = \O(R^{\a + \b+\ct+\d-1} \log^{\a+\b+\ct+\d}(R) ) =\mathrm{o}(1).
 \]
 Therefore, we can apply the global divergence theorem with the vector field $X=\mathrm{grad}_g (u)$ and conclude that
 \[
 \int_{M} \Delta_g u \, \Id\mu_{g}=\int_{M} \mathrm{div}_g\circ \mathrm{grad}_g (u) \, \Id\mu_{g} =0,
 \]
completing the proof.
\end{proof}

\section{Proof of the main result}\label{beweis}

\subsection{A collection of preliminary results}
The proof of Theorem \ref{est} will be prepared with several auxiliary results. 

\begin{Notation} Given an open subset $\Omega\subset \IR^m$, a function $f:\Omega\to\IR$, and $0<\alpha\leq 1$, we denote with
$$
[f]_{0,\alpha;\Omega}:=\sup_{x,y\in \Omega, x\ne y}\f{|f(x)-f(y)|}{|x-y|^{\alpha}}\in [0,\infty]
$$
the $\alpha$-Hölder constant of $f$.
\end{Notation}

Note that for every $0<\alpha\leq \beta\leq 1$ one has the trivial estimate
$$
[f]_{0,\alpha;\Omega}\leq \mathrm{diam}(\Omega)^{\beta-\alpha} [f]_{0,\beta;\Omega}.
$$

We record the following interior elliptic estimate from \cite{HPW}, which  follows from bootstrapping classical interior estimates:

\begin{Lemma}\label{him2} Let $P$ be a second order smooth elliptic differential operator of the form $P=\sum_{ij}a^{ij}\partial_i \partial_j + b$ which is defined on a Euclidean ball $B^{\mathrm{eucl}}_{2}(0)\subset \IR^m$. Assume that 
\begin{itemize}
\item $(a^{ij})\geq 1/4$ as a bilinear form
\item for all $i,j$, the function $a^{ij}$ is Lipschitz continuous
\item $b$ is bounded.
\end{itemize}
Pick $\Lambda>0$, $0<\alpha\leq 1$ such that
$$
\max_{ij}[a^{ij}]_{0,\alpha;B^{\mathrm{eucl}}_{1}(0)}\leq  \Lambda,\>\>\max_{ij}\left\|a^{ij}\right\|_{\mathsf{L}^{\infty}(B^{\mathrm{eucl}}_{2}(0))}\leq \Lambda, \>\>\left\|b\right\|_{\mathsf{L}^{\infty}(B^{\mathrm{eucl}}_{2}(0))}\leq \Lambda.
$$
Then there is a constant $C=C(m,\Lambda,\alpha)>0$ which only depends on the indicated parameters, such that, for any $u\in\mathsf{W}^{2,q}(B^{\mathrm{eucl}}_{2}(0))$ one has 
$$
\left\|u\right\|_{\mathsf{W}^{2,q}(B^{\mathrm{eucl}}_{1}(0))}\leq C  \left\|Pu\right\|_{\mathsf{L}^{q}(B^{\mathrm{eucl}}_{2}(0))}+C\left\|u\right\|_{\mathsf{L}^{2}(B^{\mathrm{eucl}}_{2}(0))}.
$$
\end{Lemma}

Next we record the following Morrey-type inequality that can be  deduced e.g. from \cite[Theorem 7.19]{GT}:

\begin{Lemma}\label{lemma-Morrey}
Let $u : \Omega \to \rr$ be a $\mathsf{W}^{1,p}$-function on the bounded domain $\Omega \subset \rr^{m}$. Assume that
\[
\max_{k} \| \partial_{k}u \|_{\IL^{p}(\Omega)} \leq K,
\]
for some constants $K>0$ and $p>m$. Then, $u \in \mathsf{C}^{0,\alpha}(\Omega)$ with $\alpha = (p-m)/p$ and, for all open balls $B^{\mathrm{eucl}}_{2R} \subset \Omega$, we have the estimate
\begin{equation}\label{Mor1}
[u]_{0,\alpha;B^{\mathrm{eucl}}_{R}} \leq CK,
\end{equation}
where  $C=C(m,\alpha)>0$ is a universal constant.
\end{Lemma}

\begin{Notation} In the sequel, given $(M,g)$ a smooth Riemannian manifold of dimension $m$, for every $p\in (m,\infty)$, $Q>1$, the symbol $r_{p,Q,g}(x)$ will stand for the $\mathsf{W}^{1,p}$-harmonic radius at $x$ with accuracy $Q$. In other words, $r_{p,Q,g}(x)$ is defined to be supremum of all $r>0$ such that $B_g(x,r)$ admits a centered $\Delta_g$-harmonic chart 
$$
\phi=(x_1\dots, x_m):B_g(x,r)\longrightarrow U\subset \IR^m
$$
which satisfies
$$
Q^{-1}(\delta^{ij})\leq (g^{ij})\leq Q(\delta^{ij}) \ ,\>\>\>r^{1-m/p}\max_{i,j,k}\left\|\partial_k g^{ij}\right\|_{\IL^p(U)}\leq Q-1
$$
as bilinear forms, where by \emph{centered} we mean $\phi(x)=0$.

\end{Notation}

\begin{Notation}
Let $a\vee b:=\max(a,b)$, $a\wedge b:=\min(a,b)$ for real numbers $a,b$. 
\end{Notation}

We shall also borrow the following precise $\mathsf{W}^{1,p}$-harmonic radius estimates from \cite{cheeger-a} (see also \cite{HPW}), which at least in a qualitative form can be traced back to \cite{jost}:

\begin{Lemma}\label{cheeger} For every natural $m\geq 2$, $p\in (m,\infty)$, $Q>1$, there exists a constant $A=A(m,p,Q)>0$, which only depends on the indicated parameters, such that for all $C>0$, all smooth Riemannian $m$-manifolds $(M,g)$ with $\mathrm{Ric}_g\geq -1/C^2$, and all $x\in M$ one has
$$
1\wedge r_{p,Q,g}(x)\geq A \cdot (r_{\mathrm{inj},g}(x) \wedge 1 \wedge C).
$$
\end{Lemma}

Finally, we point out the following well known localised version of the Cartan-Hadamard theorem:

\begin{Lemma}\label{localisometry} Let $(M,g)$ be a smooth complete Riemannian manifold of dimension $m\geq 2$. Assume that $x\in M$, $R_0$, $K>0$, are chosen such that 
$$
\mathrm{Sec}_g\leq K\quad\text{ on $B_g(x,R_0)$}.
$$
Then, for every $0<R < (\pi/\sqrt{K})\wedge R_0$,  there exist a smooth complete Riemannian manifold $(\bar M , \bar g)$, a point $\bar x \in \bar M$, and a smooth surjective local isometry 
$$
F=F_{g,x,R} : B_{\bar g} (\bar x ,R) \longrightarrow B_{g}(x,R)
$$
such that:\vspace{2mm}

\noindent $\mathrm{(a)}$ $F(\bar x) = x$ and $\mathrm{(b)}$ $r_{\mathrm{inj}, \bar g}(\bar x) \geq R$. \vspace{2mm}

\noindent In particular, by the local isometric property of $F$, it holds:\vspace{2mm}

\noindent $\mathrm{(c)}$ $\sect_{\bar g} \leq K$ on $B_{\bar g}(\bar x, R)$ and $\mathrm{(d)}$ $F(B_{\bar g} (\bar x ,r)) = B_{g}(x,r)$, for every $0< r \leq R$.
\end{Lemma}

For the sake of completeness, a proof of Lemma \ref{localisometry} will be given in Appendix \ref{appendix-geometric}. Obviously, the original version by Cartan-Hadamard for $K=0$ corresponds to the choices $R  = +\infty$, $F = \exp_{x}$ and $(\bar M , \bar g ) = (T_{x}M , \exp_{x}^{\ast} g)$, which is a smooth complete Riemannian manifold of non-positive curvature with $r_{\mathrm{inj},\bar g}(0)= +\infty$. The local version is obtained by cutting this construction. A version of Lemma \ref{localisometry} valid for a possibly incomplete $(M,g)$ could be obtained by assuming that the exponential map $\exp_{x}$ is well-defined on a sufficiently large ball around $x$, however the statement would look quite artificial. Finally, we remark that the result is local in the sense that the infinity of $(\bar M ,\bar g)$ plays no role and can be prescribed (almost) arbitrarily. For instance, we can (and do)  assume that $\bar M$ is diffeomorphic to $\rr^{m}$ and isometric to the standard Euclidean space outside $B_{\bar g}(\bar x, R+\epsilon)$, $0 < \epsilon \ll 1$.

\subsection{The proof of Theorem \ref{est}}
After these preparations we can finally give the proof of the main theorem of the paper: 

\begin{proof}[Proof of Theorem \ref{est}] We set $f:=\Delta _g \psi$. From now on we fix an arbitrary $x\in M$, and numbers $R_{0}$, $A_1$, $A_2$ as in the assumptions. 
Set
$$
0<R:=R_{0} \wedge (\pi/\sqrt{A_{1}})/2 .
$$

According to Lemma \ref{localisometry} there exists a pointed smooth complete Riemann manifold $(\bar M ,\bar g,\bar x)$, and a smooth surjective local isometry
$$
F=F_{g,x,R}:B_{\bar g}(\bar x,R) \longrightarrow B_{g}(x,R)
$$
such that
\begin{align}\label{mes}
r_{\mathrm{inj},\bar g} (\bar x) \geq R,\>\>
 \mathrm{Sec}_{\bar g}  \leq A_1,\>\>\mathrm{Ric}_{\bar g} \geq - A_2\text{ on }B_{\bar g}(\bar x,R).
\end{align}
We define $\bar \psi:= \psi\circ F  \in \ICC(B_{\bar g}(\bar x,R))$,  $\bar f:= f\circ F  \in \ICC(B_{\bar g}(\bar x,R))$, which, as $F$ is a smooth local isometry, satisfy
\[
\Delta_{\bar g} \bar \psi= \bar f\>\>\text{ in $ B_{\bar g}(\bar x,R)$}.
\]
We shall now estimate the derivative of $\bar \psi$ instead of $\psi$ and then \lq\lq{}pushforward\rq\rq{} the estimate carefully to $\psi$. To this end, set
$$
0<\bar r:=(r_{m+1,4,\bar g}(\bar x)\wedge R \wedge 1)/2<\infty
$$
and pick a centered $\Delta_{\bar g}$-harmonic chart 
$$
\phi=(y_1\dots, y_m):B_{\bar g}(\bar x, \bar r)\longrightarrow U\Subset \IR^m
$$
which satisfies
\begin{align}\label{gap}
(1/4)(\delta^{ij})\leq (\bar g^{ij})\leq 4(\delta^{ij}) \ ,\>\>\>\bar r^{1-m/(m+1)}\max_{i,j,k}\left\|\partial_k \bar g^{ij}\right\|_{\IL^{m+1}(U)}\leq 1.
\end{align}
The latter $\IL^{m+1}$-bound and Lemma \ref{lemma-Morrey} applied with $K = \bar r^{-1+m/(m+1)}$ and $p=m+1$ imply that there exists a constant $C_1=C_1(m)>0$ such that, for all open Euclidean balls $B^{\mathrm{eucl}}_{2R}\subset U$, one has
$$
\max_{ij}[\bar g^{ij}]_{0,1/(m+1);B^{\mathrm{eucl}}_{R}}\leq C_1 \bar r^{-1+m/(m+1)}.
$$
In addition, as $\phi$ is $\Delta_{\bar g}$-harmonic, it follows that in $B_{\bar g}(0,\bar r)$ one has $\Delta_{\bar g}= \sum_{ij}\bar g^{ij}\partial_i\partial_j$. Next note that 
$$
B^{\mathrm{eucl}}_{ 2\bar r /8}(0) \Subset U.
$$
We want to apply Lemma \ref{him2}. To this end we scale
 $$
\tilde{g}^{ij}(y):=\bar g^{ij}( y\bar r 8^{-1}),\>\>\tilde{f}(y):=\bar f (y\bar r 8^{-1} ),\>\>\tilde{\psi}(y):=\bar \psi  (y\bar r 8^{-1} ),\>\>\text{ $y\in B^{\mathrm{eucl}}_{2}(0)$},
$$
and define $\tilde{P}:=\sum_{ij}\tilde{g}^{ij}\partial_i\partial_j$. Then, noting
\begin{align*}
&\tilde{P}\tilde{\psi}= \bar r^2 (64)^{-1}\tilde{f},\\
&\max_{ij}[\tilde{g}^{ij}]_{0,1-m/(m+1);B^{\mathrm{eucl}}_2(0)}\leq (\bar r 8^{-1})^{1-m/(m+1)}\max_{ij}[\bar g^{ij}]_{0,1-m/(m+1);B^{\mathrm{eucl}}_{2\bar r 8^{-1}}(0)} \leq C_1,\\ 
&\bar r\leq 1,
\end{align*}
Lemma \ref{him2} applied to $P:=\tilde{P}$, $q:=m+1$, $u:=\tilde{\psi}$, $\Lambda:=4\vee C_1$, implies that for some $C_4=C_4(m)>0$, $C_5=C_5(m)>0$ one has
$$
\sup_{B^{\mathrm{eucl}}_{1}(0)}\sqrt{\sum_j\left|\partial_j \tilde{\psi}\right|^2}\leq C_4\left\|\tilde{\psi}\right\|_{\mathsf{W}^{2,m+1}(B^{\mathrm{eucl}}_{2}(0))}\leq C_5  \left\|\tilde{f}\right\|_{\mathsf{L}^{m+1}(B^{\mathrm{eucl}}_{2}(0))}+C_5\left\|\tilde{\psi}\right\|_{\mathsf{L}^{2}(B^{\mathrm{eucl}}_{2}(0))},
$$
where we have used Sobolev\rq{}s embedding theorem for the first inequality. Rescaling everything shows that  
\begin{align*}
&\f{\bar r}{8} \sup_{B^{\mathrm{eucl}}_{\bar r 8^{-1}}(0)}\sqrt{\sum_j\left|\partial_j \bar{\psi}\right|^2}\leq  \sup_{B^{\mathrm{eucl}}_{1}(0)}\sqrt{\sum_j\left|\partial_j \tilde{\psi}\right|^2}\\
&\leq   C_5     8^{m/(m+1)}r^{-m/(m+1)}\left\|\bar f\right\|_{\mathsf{L}^{m+1}(B^{\mathrm{eucl}}_{\bar r 4^{-1}}(0))}+C_58^{m/(m+1)}\bar r ^{-m/2}\left\|\bar \psi\right\|_{\mathsf{L}^{2}(B^{\mathrm{eucl}}_{\bar r 4^{-1}}(0))}\\
&\leq  C_58^{m/(m+1)}\bar r ^{-m/(m+1)} |B^{\mathrm{eucl}}_{\bar r 4^{-1}}(0)|^{\frac{1}{m+1}} \left\| \bar f\right\|_{\mathsf{L}^{\infty}(B^{\mathrm{eucl}}_{\bar r 4^{-1}}(0))}\\
&+C_5 8^{m/(m+1)}\bar r ^{-m/2}|B^{\mathrm{eucl}}_{\bar r 4^{-1}}(0)|^{\frac{1}{2}}\left\| \bar\psi\right\|_{\mathsf{L}^{\infty}(B^{\mathrm{eucl}}_{\bar r 4^{-1}}(0))}\\
&= C_5 2^{m/(m+1)}\left\|\bar f\right\|_{\mathsf{L}^{\infty}(B^{\mathrm{eucl}}_{\bar r 4^{-1}}(0))}+C_5 2^{m/(m+1)}\left\| \bar\psi\right\|_{\mathsf{L}^{\infty}(B^{\mathrm{eucl}}_{\bar r 4^{-1}}(0))}\\
&\leq C_5 2^{m/(m+1)}\left\|\bar f\right\|_{\mathsf{L}^{\infty}(B_{\bar g}(\bar x,\bar r /2))}+C_5 2^{m/(m+1)}\left\|  \bar\psi\right\|_{\mathsf{L}^{\infty}(B_{\bar g}(\bar x,\bar r /2))},
\end{align*}
so that using that $F$ is a smooth local isometry with $F(\bar x)=x$ and the first inequality in (\ref{gap}), for some $C_6=C_6(m)>0$, $C_7=C_7(m)>0$,
\begin{align}\nn
&|\Id  \psi(x)|_g=|\Id  \bar\psi(0)|_{\bar g}\leq \sup_{B^{\mathrm{eucl}}_{\bar r 8^{-1}}(0)}| \Id \bar \psi|_{\bar g}\leq  C_6\sup_{B^{\mathrm{eucl}}_{\bar r 8^{-1}}(0)}\sqrt{\sum_j\left|\partial_j \bar{\psi}\right|^2} \\ \nn
& \leq C_7 \bar r^{-1}\big(\left\|\bar f\right\|_{\mathsf{L}^{\infty}(B_{  \bar g}(\bar x,\bar r /2))}+\left\|\bar \psi\right\|_{\mathsf{L}^{\infty}(B_{ \bar g}(\bar x,\bar r /2))}\big)\\ \nn
&= 2 C_7 (r_{m+1,4,\bar g}(\bar x)\wedge 1 \wedge R)^{-1}\big(\left\|\bar f\right\|_{\mathsf{L}^{\infty}(B_{\bar g}(\bar x,\bar r /2))}+\left\| \bar\psi\right\|_{\mathsf{L}^{\infty}(B_{  \bar g}(\bar x,\bar r /2))}\big)\\ \nn 
&\leq  2 C_7 (r_{m+1,4,\bar g}(\bar x)\wedge 1 \wedge R)^{-1}\big(\left\| f\right\|_{\mathsf{L}^{\infty}(B_{ g}( x,R /2))}+\left\| \psi\right\|_{\mathsf{L}^{\infty}(B_{  g}(x,R /2))}\big),\\ \nn
&\leq  2 C_7 (r_{m+1,4,\bar g}(\bar x)\wedge 1 \wedge R)^{-1}\big(\left\|f\right\|_{\mathsf{L}^{\infty}(B_{ g}(x,R /2))}+\left\| \psi\right\|_{\mathsf{L}^{\infty}(B_{  g}(x,R /2))}\big),
\end{align}
as $\bar r< R$. Finally, by Lemma \ref{cheeger} and (\ref{mes}) we can pick a constant $C_8=C_8(m)>0$ such that
$$
r_{m+1,4,\bar g}(\bar x)\wedge 1  \wedge R\geq C_8 \cdot R \wedge 1\wedge \sqrt{1/A_2}.
$$
This completes the proof of the pointwise inequality
\begin{align}\label{ppintwise-estimate}
|\Id \psi(x)|_g \leq \frac{C}{\min \big( 1, R , \sqrt{1/A_2} \big)} \Big(\sup_{ B_g(x, R )} \left|\Delta_g \psi\right|+\sup_{ B_g(x,R)}\left| \psi\right|\Big)
\end{align}
with  $R = R_{0} \wedge (\pi/\sqrt{A_{1}})/2$. Whence, to conclude the validity of the uniform estimate \eqref{uniform-estimate}, we note that if $y \in B_{g}(x,R_0/4)$, by the obvious inclusion $B_{g}(y,R_0/2) \subset B_{g}(x,R_{0})$, we have that the curvature conditions \eqref{curv-conditions2} still hold on $B_{g}(y,R_0/2)$. Therefore, the result follows by applying \eqref{ppintwise-estimate} on each $B_{g}(y,R_0/2)$.
\end{proof}

\appendix

\section{Proof of Lemma \ref{localisometry}}\label{appendix-geometric}
For the sake of completeness we include a proof of Lemma \ref{localisometry}.

The following notation will be used repeatedly: given a smooth Riemannian manifold $(M,g)$, a point $x\in M$ and a number $C>0$ we denote the centered Euclidean tangential ball as
$$
\IBB_g(x,C):= \{v|\>v\in T_{x}M\>, |v|_g< C\} \subset T_{x}M .
$$
Thus, if $(M,g)$ is complete, for every $C>0$ and $x \in M$ we have that 
$$
\exp_{g,x}|_{\IBB_{g}(x,C)} : \IBB_{g}(x,C) \longrightarrow B_{g}(x,C)
$$
is a smooth surjective map.\smallskip

Now, let $(M,g)$ be as in the statement of the lemma. By scaling $g$ we can assume $K=1$. By Rauch comparison, the conjugate radius of $(M,g)$ satisfies $\mathrm{conj}_g(x) \geq \pi$. Thus, for every $0<R < \pi\wedge R_0$, the exponential map,
$$
 \exp_{g,x}|_{\IBB_{g}(x,R)} :  T_{x}M \supset \IBB_{g}(x,R)\longrightarrow B_{g}(x, R)\subset M,
$$
is a smooth local diffeomorphism, hence altogether a smooth surjective local isometry with respect to the pull-back metric
$$\bar g_{R} = \exp_{x}^{\ast}g.$$
In particular,  $\sect_{\bar g_{R}} \leq 1$.
Note also that, in the identification $T_{0} T_{x}M = T_{x}M$, since $\Id_{0} \exp_{g,x} = \mathrm{id}$, we have
\[
(\bar g_{R})_{0} = g_{x}.
\]
Fix $0< R < \pi\wedge R_0$ and extend the Riemannian metric $\bar g_{R}$ outside $\IBB_{g}(x,R)$ so to obtain a smooth complete Riemannian manifold  $\bar M = (T_{x}M , \bar g)$. For instance, using a partition of unity, we can glue $\bar g_{R+2\epsilon}$ on $\IBB(x,R+2\epsilon)$, $R+2\epsilon < \pi \wedge R_{0}$, with $g_{\rr^{m}}$ on $T_{x}M \setminus \IBB(x,R+\epsilon )$.
Since
\[
\bar g = \bar g_{R} \, \text{ on } \IBB_{g}(x,R)\subset \bar M
\]
then
\[
\sect_{\bar g} \leq 1\, \text { on } \IBB(x, R) \subset \bar M.
\]
Moreover, for every $v \in T_{x}M = T_{0}T_{x}M$ with $|v|_{g_{x}} = |v|_{\bar g_{0}} = 1$, the line segment
$$\bar \gamma_{v}(t) = tv : [0,R) \to   \IBB(x,R)\subset \bar M$$
is the (unique) geodesic of $\bar M$ issuing from $0 \in \bar M$ with unit speed $v \in T_{0} \bar M = T_{x}M$.
Indeed, $\bar \gamma_{v}$ projects to the geodesic 
$$
\gamma_{v}(t) = \exp_{g,x} (tv): [0,R) \longrightarrow B_{g}(x,R) 
$$
of $M$ via the smooth locally isometric map $\exp_{x}|_{\IBB_{g}(x,R)}$.
In particular, no such two distinct geodesics $\bar \gamma_{v_{1}}, \bar \gamma_{v_{2}}$ can intersect inside $\IBB_{g}(x,R)$ away from $0$. Furthermore, the curvature condition on $\IBB_{g}(x,R) \subset \bar M$ implies that no point of $\IBB_{g}(x,R) \setminus \{0\}$ is conjugate to $0$ along $\bar \gamma_{v}$. It follows that $\bar \gamma_{v}(t)$ is minimizing in $\bar M$ for every $t \in [0, R)$ or, equivalently,
\[
r_{\mathrm{inj},\bar g}(0) \geq R.
\]
Note that, by the local isometry property of $\exp_{g,x}|_{\IBB_{g}(x,R)}$, for every $v \in \IBB_{g}(x,R)\subset T_{0}\bar M= T_{x}M$ the geodesic 
$$
\bar \gamma_{v} : [0,1] \longrightarrow \IBB_{g}(x,R)\subset \bar M
$$
 has length
\[
 \ell_{\bar g} (\bar \gamma_{v}) = \ell_{g} (\gamma_{v}) = |v|_{g}.
\]
Since $\bar \gamma_{v}$ is $\bar g$-minimizing, it follows
\[
\Id_{\bar g}(v,0) =  |v|_{g}, \, \text {on } \IBB_{g}(x,R),
\]
proving the inclusion $\IBB_{g}(x,R) \subseteq B_{\bar g}(0,R)$. \\
On the other hand, if $v \in \bar M \setminus \IBB_{g}(x,R)$, take a unit speed minimizing geodesic 
$$
\bar \gamma : [0, \bar d] \longrightarrow \bar M
$$
such that $\bar \gamma(0) = 0$ and $\bar \gamma (\bar d) = v$. Let $w = \dot {\bar \gamma} (0)$ and recall that $|w|_{\bar g_{0}} =  |w|_{g_{x}}$. Then, $|w|_{g_{x}} =1$ and $\bar \gamma_{w}(t) \in \IBB_{g}(x,R)$ for every $0 \leq t <R$. By uniqueness, we deduce that $\bar \gamma (t) = \bar \gamma_{w}(t)$ for every $0 \leq t < R$. This implies
\[
\bar d = \Id_{\bar g}(v,0) \geq \ell_{\bar g}(\bar \gamma_{w}|_{[0,R)}) = R.
\]
It follows that
\[
B_{\bar g}(0,R) = \IBB_{g}(x,R)
\]
and the Lemma is proved with $F_{g,x,R}: = \exp_{g,x}|_{\IBB_{g}(x,R)}$ and $\bar x: =0$.

\end{document}